\documentclass[11pt, reqno]{amsart}
\usepackage[T1]{fontenc}
\usepackage{yhmath,amsmath,amsfonts,amssymb,enumerate,amsthm,appendix,caption,lipsum,}
\usepackage{enumitem}
\usepackage[english]{babel}
\usepackage{cmap,mathtools,bm}
\usepackage{graphics} 
\usepackage{epsfig} 
\usepackage{graphicx}\usepackage{epstopdf}
\usepackage{a4wide}
\usepackage[margin=0.90in]{geometry}

\usepackage{xcolor}
\usepackage{hyperref}
\usepackage[hyperpageref]{backref}
\hypersetup{
    colorlinks=true,
    linkcolor=myblue,
    filecolor=magenta,      
    urlcolor=mygreen,
    citecolor=mygreen,
}
\definecolor{mygreen}{rgb}{0.01,0.6,0.2}
\definecolor{myblue}{rgb}{0.01, 0.18, 1.0}
\newtheorem{theorem}{Theorem}
\newtheorem{proposition}[theorem]{Proposition}

\theoremstyle{definition}

\newtheorem{remark}[theorem]{Remark}
\newtheorem{example}[theorem]{Example}

\numberwithin{equation}{section}
\numberwithin{theorem}{section}
\numberwithin{equation}{section}
\numberwithin{theorem}{section}

\newcommand{\al} {\alpha}
\newcommand{\pa} {\partial}
\newcommand{\be} {\beta}
\newcommand{\de} {\delta}
\newcommand{\De} {\Delta}

\newcommand{\ga} {\gamma}
\newcommand{\Ga} {\Gamma}
\newcommand{\om} {\omega}
\newcommand{\Om} {\Omega}
\newcommand{\la} {\lambda}

\newcommand{\no} {\nonumber}
\newcommand{\noi} {\noindent}
\newcommand{\ep} {\epsilon}

\newcommand{\ra} {\rightarrow}

\def\t{\tau}

\DeclarePairedDelimiter\abs{\lvert}{\rvert}%
\DeclarePairedDelimiter\norm{\lVert}{\rVert}%
\makeatletter
\let\oldnorm\norm
\def\norm{\@ifstar{\oldnorm}{\oldnorm*}}

\newcommand{\hst}{{\dot{H}}^s(\RN)}
\newcommand{\intRn}{\displaystyle{\int_{\mathbb{R}^N}}}

\newcommand\restr[2]{{
  \left.\kern-\nulldelimiterspace 
  #1 
  \right|_{#2} 
  }}

\def\t{\tau}

\def\C{{\mathcal C}}

\def\N{{\mathbb N}}
\def\F{{\mathcal F}}

\def\R{{\mathbb R}}
\def\RN{{\mathbb R}^N}

\def\({{\Big(}}
\def\){{\Big)}}

\def\ws2{{\F_{\frac{N}{2}}}}

\def\c1{{\C_c^1}}

\def\dr{{\rm d}r}

\def\dz{{\rm d}z}
\def\dx{{\rm d}x}
\def\dy{{\rm d}y}

\def\R{{\mathbb R}}

\def\C{{\mathcal C}}

\def\R{{\mathbb R}}
\def\N{{\mathbb N}}
\def\F{{\mathcal F}}

\def\ws2{{\F_{\frac{N}{2}}}}

\def\l2{{ L^{1,\;\infty}(\log L)^2}}
\usepackage[foot]{amsaddr}

\usepackage[pagewise]{lineno}

\title[Study of fractional semipositone problems on $\RN$]{Study of fractional semipositone problems on $\RN$}

\author[N. Biswas]{Nirjan Biswas}


\email{nirjan22@tifrbng.res.in, nirjaniitm@gmail.com}

\begin{document}

\maketitle

\centerline{Tata Institute of Fundamental Research, Centre For Applicable Mathematics,}
\centerline{Post Bag No 6503, Sharada Nagar,}
\centerline{Bangalore 560065, India}

\begin{abstract} 
Let $s \in (0,1)$ and $N >2s$. In this paper, we consider the following class of nonlocal semipositone problems:
\begin{align*}
    (-\Delta)^s u=  g(x)f_a(u)  \text {  in  }  \mathbb{R}^N, \; u > 0 \text{ in } \mathbb{R}^N, 
\end{align*}
where  the weight $g \in L^1(\mathbb{R}^N) \cap L^{\infty}(\mathbb{R}^N)$ is positive, $a>0$ is a parameter, and $f_a \in \C(\mathbb{R})$ is strictly negative on $(-\infty,0]$. For $f_a$ having subcritical growth and weaker Ambrosetti-Rabinowitz type nonlinearity, we prove that the above problem admits a mountain pass solution $u_a$, provided `$a$' is near zero. To obtain the positivity of $u_a$, we establish a Brezis-Kato type uniform estimate of $(u_a)$ in $L^r(\mathbb{R}^N)$ for every $r \in [\frac{2N}{N-2s}, \infty]$.
\end{abstract} 

\section{Introduction }\label{intro}
In this present paper, we deal with a class of nonlocal semipositone problems on $\RN$. Precisely, for $s \in (0,1)$ and $N >2s$, we consider the following nonlocal semilinear equation:  
\begin{align}\label{SP}\tag{SP}
(-\De)^s u=  g(x)f_a(u)  \text {  in  }  \RN, 
\end{align}
where the weight $g \in L^1(\RN) \cap L^{\infty}(\RN)$ is positive and satisfies \ref{g1}, $a>0$ is a parameter, and $f_a \in \C(\mathbb{R})$ has the following form: 
\begin{align*}
    f_a(t) = \left\{\begin{array}{ll} 
            f(t)-a , & \text {if }  t \ge 0; \\ 
            -a, & \text{if} \; t \le 0,  \\
             \end{array} \right. \text{ with } f \in \C(\R^+) \text{ satisfying } f(0)=0.
\end{align*}
Moreover, $f$ satisfies the following hypothesis: 
\begin{enumerate}[label={($\bf f{\arabic*}$)}]
\setcounter{enumi}{0}
    \item \label{f1} $\displaystyle \lim_{t \rightarrow 0}\frac{f(t)}{t}=0$,  $\displaystyle \lim_{t \rightarrow \infty}\frac{f(t)}{t}=\infty$, and $\displaystyle \lim_{t \rightarrow \infty}\frac{f(t)}{t^{\ga-1}} \le C(f)$ for some $\gamma\in (2, 2^*_s)$ and $C(f)>0$, \vspace{0.2 cm}
    \item \label{f2} there exists  $R>0$ such that  $\displaystyle \frac{f(t)}{t}$ is increasing for $t>R$,
\end{enumerate}
where $2^*_s = \frac{2N}{N-2s}$ is the critical fractional exponent.  
The linear operator $(-\De)^{s}$ is called the fractional Laplacian defined as
$$
(-\Delta)^{s} u(x):=2 \lim_{\epsilon \rightarrow 0^{+}} \int_{\mathbb{R}^{N} \backslash B_{\epsilon}(x)} \frac{u(x)-u(y)}{|x-y|^{N+2s}} \,\dy, \;\;\;x\in \RN,
$$
where $B_{\epsilon}(x)$ denotes the ball of radius $\epsilon$ and centred at $x$. Due to the presence of the strictly negative quantity on the R.H.S. of \eqref{SP} in the regions where $u \le 0$ and certain portion of $u> 0$, the problem \eqref{SP} is called \textit{semipositone} in the literature. Semipositone problems have applications in mathematical physics, biology, engineering etc. More preciously, in the logistic equation, mechanical systems, suspension bridges, population model, etc.; see for example \cite{JS2010, OSS02}. 

In the local case, the semipositone problems were first observed by Brown and Shivaji in \cite{BS83} while studying the perturbed bifurcation problem  $-\Delta u =\lambda(u-u^3) -\epsilon$ in $\Om$, $u > 0$ in $\Om$, $u=0$ on $\pa \Om$, where $\la, \ep >0$ and $\Om$ is a bounded domain. In this work, the authors used the sub-super solution method to get positive solutions. Observe that $u=1$ is a supersolution for this problem since the R.H.S. of the equation is negative at $u=1$. To obtain an appropriate positive subsolution, the authors used the anti-maximum principle due to Cl\'{e}ment and Peletier. Later, many authors studied the following semipositone problem on a bounded domain $\Om$:
\begin{align}\label{GSP}
-\Delta u = \lambda f(u)  \text {  in  }  \Om, \; u > 0 \text {  in  }  \Om, \, \text{ and } \, u=0  \text{  on  } \pa \Om,
\end{align}
where $\la >0$, $f:\R^+ \ra \R$ is continuous, increasing and $f(0)<0$. For example, we refer \cite{CHS95, CS88, CS89, CDS15, DZ05} where various growth conditions and nonlinearities on the function $f$ are imposed to find the existence of positive solutions for \eqref{GSP}. 
In \cite{AHS20}, Alves et al. considered the following semipositone problem:
\begin{align}\label{Alves}
    -\Delta u = g(x)f_a(u)  \text {  in  }  \mathbb{R}^N, 
\end{align}
with $f_a(t)= f(t)-a$ for $t>0$, $f_a(t)=-a(t+1)$ for $t \in [-1,0]$, and $f_a(t)=0$ for $t \le -1$, where the function $f \in \C(\R^+)$ satisfies $f(0)=0$, is locally Lipschitz, has superlinear growth conditions and Ambrosetti-Rabinowitz (see \cite{AR1973}) type nonlinearities. Meanwhile, the weight function $g$ is assumed to be positive, radial, lies in $L^1(\RN) \cap L^{\infty}(\RN)$ and satisfies the following bound: 
\begin{align}\label{Alves1}
 |x|^{N-2} \intRn \frac{g(|y|)}{|x-y|^{N-2}} \, \dy \le C(g), \quad \text{for } x \in \RN \setminus \{ 0 \}, \text{ where } C(g)>0.
\end{align}
The authors used the regularity estimate by Brezis and Kato in \cite{BK1979} and the Reisz potential for the Laplace operator to establish uniform boundedness of mountain pass solutions of \eqref{Alves} in $L^{\infty}(\RN)$, with respect to the parameter `$a$' near zero. The authors then obtained a positive solution of \eqref{Alves} using this uniform regularity estimate, the strong maximum principle, and the condition \eqref{Alves1}. For more results related to semipositone problems, we refer \cite{CFL16, CQT17} and the references therein. 

Now, we shift our discussion to the nonlocal case. In the past years, a significant amount of attention has been given to the study of fractional laplacian due to its numerous applications in mathematical physics, engineering and related fields. For example, in linear drift-diffusion equations \cite{CI2017}, image denoising \cite{GH2015}, quasi-geostrophic flows \cite{Barbi_2022}, and bound-state problems  \cite{SB1951}, to name a few.  Several authors recently studied the nonlocal semipositone problems on a bounded domain $\Om$. In \cite{DT21}, the authors considered a multiparameter fractional semipositone problem  $(-\De)^s u=\la (u^q-1) +\mu u^r$ in $\Om$, $u>0$ in $\Om$, $u=0$ in $\RN \setminus \Om$, where  $\la, \mu >0$ are parameters, $N>2s$, and $0<q<1<r\le 2^*_s -1$. For a certain range of $\la$ and $\mu$, the authors proved the existence of positive solution for this problem. Their proof relies on the construction of a positive subsolution. Later, in \cite{LLV2022}, the authors studied the nonlocal nonlinear semipositone problem $(-\De)_p^s u = \la f(u)$ in $\Om$, $u=0$ in $\RN \setminus \Om$, where $(-\De)_p^s$ is the fractional $p$-Laplace operator, $\la>0$, $f \in \C(\R)$ has superlinear, subcritical growth and $f(s)=0$ for $s \le -1$. The authors obtained at least one positive solution provided the parameter $\la$ is sufficiently small. Their proof uses regularity results up to the boundary of $\Om$ and Hopf’s Lemma for $(-\De)_p^s$. To our knowledge, nonlocal semipositone problems on an unbounded domain have not been studied yet.

In this paper, we consider the nonlocal counterpart of $-\De u=  g(x)(f(u)-a)$ in $\RN$, $u>0$ in $\RN$ (studied in \cite{AHS20}). On the weight function $g$, we impose a nonlocal analogue of \eqref{Alves1}. With subcritical, superlinear and without Ambrosetti-Rabinowitz growth conditions (see \cite{AR1973}) for $f$, our primary concern is to establish the existence of positive solution to \eqref{SP}, depending on the parameter `$a$'. The techniques used in \cite{DT21, LLV2022} to get positive solution are not adoptable in this context. Our procedure to find non-negative solution for \eqref{SP} is motivated by \cite{AHS20}, where the uniform regularity estimate (with respect to `$a$') of solutions in $L^{\infty}(\RN)$ plays a major role. In \cite[Proposition 5.1.1]{DMV17}, for $g \in L^1(\RN) \cap L^{\infty}(\RN)$, and $\abs{f(x,t)} \le C(1+ \abs{t}^p); 1 \le p \le 2^*_s-1$, the authors proved that every non-negative solution to the problem $(-\De)^s u= g(x)f(x,u)$ in $\RN$ is bounded. However, this regularity result is not applicable in our situation. Also, the Brezis-Kato type regularity estimate for weak solution to \eqref{SP} is unknown.

We consider the homogeneous fractional Sobolev space ${\dot{H}}^s(\RN)$ (introduced in \cite{DMV17}) defined as 
$$\hst:= \text{closure of } \\C_c^1(\RN) \text{ with respect to } \norm{\cdot}_{\hst},$$
where $\norm{\cdot}_{\hst} := [\cdot]_{s,2} + \norm{\cdot}_{L^{2^*_s}(\RN)}$, and
\begin{align*}
    [u]_{s,2}:= \left( \; \iint\limits_{\RN \times \RN}\frac{(u(x)-u(y))^2}{|x-y|^{N+2s}}\,\dx\dy \right)^{\frac{1}{2}},
\end{align*}
is the Gagliardo seminorm. 
A function $u\in \hst$ is a weak solution of \eqref{SP} if the following identity holds:
\begin{align*}
    \iint\limits_{\RN \times \RN}\frac{(u(x)-u(y)(\phi(x) - \phi(y))}{|x-y|^{N+2s}}\,\dx\dy = \intRn g(x)f_a(u) \phi(x) \, \dx, \quad \forall \, \phi \in  \hst.
\end{align*}

We say a weak solution $u$ is a mountain pass solution of \eqref{SP} if it is a critical point of $\C^1$ energy functional associated with \eqref{SP}, which satisfies the mountain pass geometry and a weaker Palais-Smale condition (see Proposition \ref{I_a properties} and \cite[Theorem~2.1]{S1991}), and moreover, the value of the energy functional at $u$ possesses a min-max characterization (see \eqref{MP2}).  

\begin{theorem}\label{existence}
Let $s \in (0,1)$ and $N > 2s$. Assume that $f$ satisfies {\rm \ref{f1}} and {\rm \ref{f2}}. Let $g \in L^1(\RN) \cap L^{\infty}(\RN)$ be positive. Then there exists $a_1>0$ such that for each $a \in (0, a_1)$, \eqref{SP} admits a mountain pass solution. Moreover, there exist $a_2 \in (0, a_1)$ and $C>0$, such that $\norm{u_a}_{\hst} \le C$, for all $a \in (0, a_2).$
\end{theorem}

From the above theorem, observe that the mountain pass solutions of \eqref{SP} are uniformly bounded in $L^{2^*_s}(\RN)$. 
In the following theorem, we prove a Brezis-Kato type regularity result which says the uniform boundedness of the mountain pass solutions in $L^r(\RN)$ with $r \ge 2^*_s$, under an additional growth assumption on $f$ near infinity:
 
\begin{enumerate}[label={($\bf \tilde{f}{\arabic*}$)}]
\setcounter{enumi}{0}
    \item \label{f1.1} $\displaystyle \lim_{t \rightarrow \infty}\frac{f(t)}{t^{2^*_s-1}} =0$.
\end{enumerate}

\begin{theorem}\label{regularity}
   Let $s \in (0,1)$ and $N > 2s$. Let $f,g,a_2$ be as given in Theorem~\ref{existence}. In addition, assume that $f$ satisfies {\rm \ref{f1.1}}. Then for $r \in [2^*_s, \infty]$ and $a \in (0, a_2)$, $u_a \in L^r(\RN) \cap \C(\RN)$. Moreover, there exists $C(r,N,s,f,g)>0$ such that $\norm{u_a}_{L^r(\RN)} \le C$, for all $a \in (0, a_2).$ 
\end{theorem}

Next, we state the positivity of the solutions to \eqref{SP}. To state the result, we invoke further hypothesis on $f$ and $g$:
\begin{enumerate}[label={($\bf f{\arabic*}$)}]
\setcounter{enumi}{2}
    \item \label{f3} $f$ is locally Lipschitz, 
\end{enumerate}
 \vspace{- 0.2 cm}
\begin{enumerate}[label={($\bf g{\arabic*}$)}]
\setcounter{enumi}{0}
    \item \label{g1} $|x|^{N-2s} \intRn \frac{g(y)}{|x-y|^{N-2s}} \, \dy \le C(g)$, where $x \in \RN \setminus \{0\}$ and  $C(g)>0$.
\end{enumerate}
 
\begin{theorem}\label{result}
 Let $s \in (0,1)$ and $N > 2s$. Let $f,g,a_2$ be as given in Theorem~\ref{regularity}. Then the following hold: 
 \begin{enumerate}
     \item[{\rm (i)}] There exists $a_3 \in (0, a_2)$ such that for every $a \in (0, a_3)$, $u_a \ge 0$ on $\RN$. 
     \item[{\rm (ii)}] In addition, if $f$ satisfies {\rm \ref{f3}} and $g$ satisfies {\rm \ref{g1}}, then $u_a>0$ on $\RN$.
 \end{enumerate}
 \end{theorem}

Let us briefly discuss our approach to prove the above theorems. The existence of mountain pass solution is based on variational methods. To establish the Brezis-Kato type regularity result for \eqref{SP}, we use the Moser iteration technique and the uniform boundedness of the mountain pass solutions in $\hst$ along with the growth assumption \ref{f1.1}. With this regularity result and using the Riesz representation for the operator $(-\Delta)^s$ (\cite[Theorem 5]{St2019}), we show that a sequence of mountain pass solutions uniformly converges to a positive function in $\C(\RN)$ near $a=0$. In the end, we conclude the positivity of the solutions by using the properties \ref{f3} and \ref{g1}. At this point, it is clear that the range of `$a$', the growth of $f$ near infinity, and the behaviour of the weight function $g$ are essential for the existence of positive solutions. One example of $f,g$ satisfying all these properties is demonstrated in Example \ref{example1}.

We organize the rest of this paper as follows. In Section 2, we obtain some embeddings of $\hst$ and set up a variational framework  associated with \eqref{SP}. Section 3 contains the proof of the existence and regularity of the mountain pass solutions to \eqref{SP}. We establish the positivity of the solution in Section 4.

\section{Embeddings of $\hst$ and the variational settings} 

In the first part of this section, we discuss compact embeddings of $\hst$ into specific Lebesgue spaces and weighted Lebesgue spaces. Using these embeddings, in the second part, we prove qualitative properties of the energy functional associated with \eqref{SP}. In the rest of this paper, we denote $C$ as a generic positive constant, and denote the norm $\norm{\cdot}_{L^p(\RN)}$ by $\norm{\cdot}_p$.  

\subsection{Embeddings of $\hst$} 
Recall that, the homogeneous fractional Sobolev space $\hst$ is the closure of $\C_c^1(\RN)$ with respect to $[\cdot]_{s,2} + \norm{\cdot}_{2^*_s}$. 
In view of \cite[Theorem 1.1]{DV2015}, $\hst$ has the following representation: 
$$ \hst := \left\{ u : \RN \rightarrow \R \text{ is measurable} : [u]_{s,2} + \norm{u}_{2^*_s} < \infty \right\}. $$
Henceforth, $\hst \hookrightarrow L^{2^*_s}(\RN)$. Moreover, by \cite[Theorem 2.2.1]{BV2016} and using the density of $\\C_c^1(\RN)$, we get 
\begin{align}\label{Sobolev Embedding}
    \norm{u}_{2^*_s} \le C(N,s) [u]_{s,2}, \; \forall \, u \in \hst.
\end{align}
The above inequality infers that $[\cdot]_{s,2}$ is an equivalent norm in $\hst$, i.e., there exists $C_{1}$ depending on $N,s$ such that $[u]_{s,2} \le \norm{u}_{\hst} \le C_{1} [u]_{s,2}$ holds for all $u \in \hst$.
In the following proposition, we prove that $\hst$ is compactly embedded into spaces of locally integrable functions. 

\begin{proposition}\label{compact embeddings 1}
Let $N>2s$. Then $\hst \hookrightarrow L_{loc}^q(\RN)$ compactly for every $q \in (1, 2^*_s)$.
\end{proposition}
\begin{proof}
Combining the embeddings $\hst \hookrightarrow L^{2^*_s}(\RN)$ and $L_{loc}^{2^*_s}(\RN) \hookrightarrow L_{loc}^q(\RN)$ with $q \in (1,2^*_s)$, it is evident that  $\hst \hookrightarrow L_{loc}^q(\RN)$ for $q \in (1, 2^*_s)$. First, we show that $\hst$ is compactly embedded into $L_{loc}^2(\RN)$. Let $(u_n)$ be a bounded sequence in $\hst$, and $K \subset \RN$ be a compact set. Then there exists $M_1>0$ such that $\norm{u_n}_{L^2(K)} \le C \norm{u_n}_{\hst} \le M_1$ for every $n \in \N$. Consequently, the sequence $(u_n \big|_K)$ is bounded in $L^2(K)$. Further, using \cite[Lemma A.1]{BLP14} for every $n \in \N$, we have
\begin{align}\label{estimate1}
    \sup_{\abs{h}>0} \intRn \frac{(u_n(x+h)-u_n(x))^2}{\abs{h}^{2s}} \, \dx \le C(N) [u_n]_{s,2}^2.
\end{align}
The boundedness of $(u_n)$ in $\hst$ and \eqref{estimate1} confirm that for all $n \in \N$, $\int_{\RN} (u_n(x+h)-u_n(x))^2 \, dx \ra 0$ as $\abs{h} \ra 0$. Now by applying the Riesz-Fr\'{e}chet-Kolmogorov compactness theorem on $(u_n \big|_K) \subset L^2(K)$, we conclude $(u_n \big|_K)$ is relatively compact. Hence, it has a convergent subsequence in $L^2(K)$. Therefore, $\hst$ is compactly embedded into $L_{loc}^2(\RN)$. For $q \in (1,2)$, using $L_{loc}^2(\RN) \hookrightarrow L_{loc}^q(\RN)$ we directly get the compact embeddings of $\hst$ into $L_{loc}^q(\RN)$. Next, we consider $q \in (2,2^*_s)$. In this case, we express $q=2t+ (1-t)2^*_s$, where $t = \frac{2^*_s-q}{2^*_s-2} \in (0,1)$. Using $\hst \hookrightarrow L_{loc}^q(\RN)$, the sequence $(u_n \big|_K)$ is bounded in $L^q(K)$. Applying the H\"{o}lder's inequality with the conjugate pair $(\frac{1}{t}, \frac{1}{1-t})$ and \eqref{estimate1} we obtain the following estimate for every $\abs{h}>0$ and $n \in \N$:
\begin{align*}
    \intRn \abs{u_n(x+h)-u_n(x)}^q \, \dx & \le \left( \intRn (u_n(x+h)-u_n(x))^2 \, \dx \right)^t \left( \intRn \abs{u_n(x+h)-u_n(x)}^{2^*_s} \, \dx\right)^{1-t} \\
    & \le  C(N) \left( \abs{h}^s  [u_n]_{s,2} \right)^{2t}  \norm{u_n}_{2^*_s}^{2^*_s(1-t)} \le  C(N)  \abs{h}^{2st} \norm{u_n}_{\hst}^q.  
\end{align*}
Again the boundedness of $(u_n)$ in $\hst$, and the Riesz-Fr\'{e}chet-Kolmogorov compactness theorem confirm a convergent subsequence of $(u_n \big|_K)$ in $L^q(K)$. Therefore, $\hst$ is compactly embedded into $L_{loc}^q(\RN)$ for $q \in (2,2^*_s)$ as well. This completes the proof.
\end{proof}

We now prove the compact embeddings of $\hst$ into weighted Lebesgue spaces.

\begin{proposition}\label{compact embeddings 2}
Let $N>2s$ and  $q \in [2,2^*_s)$. Let $p$ be the conjugate exponent of $\frac{2^*_s}{q}$, and $g \in  L^p(\RN)$. Then the embedding $\hst \hookrightarrow L^q(\RN,|g|)$ is compact.
\end{proposition}

\begin{proof}
Let $u_n \rightharpoonup u$ in $\hst$. We need to show $u_n \ra u$ in $L^{q}(\RN,|g|)$. Set $L = \sup \{\norm{u_n-u}_{2^*_s}^{q} : n \in \N\}.$ Clearly, $L$ is finite from the boundedness of $(u_n)$ in $\hst$. Let $\ep > 0$ be given. Since $\C_c(\RN)$ is dense in $L^{p}(\RN)$, we choose $g_{\ep} \in \C_c(\RN)$ such that $\norm{g - g_{\ep} }_p < \frac{\ep}{2L}.$ We estimate
\begin{align}\label{cembed2.1}
\intRn |g| |u_n - u|^{q} \leq \intRn |g - g_{\ep}| |u_n - u|^{q} + \intRn |g_{\ep}| |u_n - u|^{q}.
\end{align}
Using the H\"{o}lder's inequality with conjugate pair $(p,\frac{2^*_s}{q})$ we estimate the first integral of the above inequality as  
\begin{align}\label{cembed2.2}
\intRn \abs{g - g_{\ep}} \abs{u_n - u}^q \le \norm{ g - g_{\ep}}_p \norm{u_n-u}_{2^*_s}^q < \frac{\ep}{2}.
\end{align} 
Suppose $K$ is the support of $g_{\ep}$. Using the compact embeddings of $\hst$ into $L^q_{loc}(\RN)$ (Proposition \ref{compact embeddings 1}), there exists $n_1 \in \N$ such that
\begin{align*}
   \intRn |g_{\ep}| |u_n - u|^{q}  = \int_{K} |g_{\ep}| |u_n - u|^{q}  < \frac{\ep}{2} , \; \forall \, n \ge n_1.
\end{align*}
Therefore, from \eqref{cembed2.1} and \eqref{cembed2.2} we obtain $$\intRn |g| |u_n - u|^{q} < \ep, \; \forall \, n \ge n_1.$$ Thus $u_n \ra u$ in $L^q(\RN,|g|)$. 
\end{proof}

\subsection{The variational settings}
For the existence of a solution of \eqref{SP}, this subsection sets up a suitable functional framework. In the following remark, we identify some bounds (upper and lower) for $f_a$ and its primitive $F_a$, defined as $F_a(t) = \int_0^t f_a(\tau) {\rm d}\tau$.     

\begin{remark}\label{bound}
(i) Let $\ep>0$ and $\ga \in (2,2^*_s]$. Using subcritical growth on $f$ and behaviour of $f$ near zero (see \ref{f1}), there exists $t_1(\ep)>0$ such that $f(t) < \ep t,$ for $0<t<t_1$, and $f(t) \le Ct^{\ga-1}$ for $t \ge t_1$, where $C=C(f,t_1(\ep))$. Hence $f(t) \le \ep t + Ct^{\ga-1}$ for $t \in \R^+$, and 
\begin{align}\label{growth 1}
    \abs{f_a(t)} \le \ep |t| + C|t|^{\ga -1} - a \text{ and }  \abs{F_a(t)} \le \ep t^2 + C|t|^{\ga} + a|t| \text{ for } t \in \R,
\end{align}
where $C=C(f,t_1(\ep))$. Again using the subcritical growth on $f$, $f(t) \le C(f) t^{\ga-1}$ for $t > t_2$. The continuity of $f$ infers that $f(t) \le C$ on $[0,t_2]$. Hence for $a \in (0, \tilde{a})$, we get
\begin{align}\label{growth 2}
    \abs{f_a(t)} \le C(1+\abs{t}^{\gamma -1}) \text{ and } \abs{F_a(t)} \le C(\abs{t} + \abs{t}^{\ga}) \text{ for } t \in \R, 
\end{align}
where $C=C(f,t_2, \tilde{a})$. Using \ref{f1.1}, there exists $t_3(\ep)>0$ so that $f(t) \le \ep t^{2^*_s-1}$ for all $t> t_3$. Hence for $a \in (0, \tilde{a})$, we also obtain 
\begin{align}\label{growth 2.1}
    \abs{f_a(t)} \le C + \ep\abs{t}^{2^*_s-1} \text{ for } t \in \R,
\end{align}
where $C=C(\ep, t_3(\ep), \tilde{a})$.

\noi (ii) Let $M>0$. Since $f$ is superlinear (see \ref{f1}), there exists a constant $C=C(M)$ such that $f(t) > Mt-C$, for every $t \in \R^+$. From the superlinearity of $f$, it also follows that $\lim_{t \rightarrow \infty}\frac{F(t)}{t^2}=\infty$, and hence $F(t)>Mt^2-C(M)$ for every $t \in \R^+$.  Accordingly, 
\begin{align}\label{growth 3}
    f_a(t) > Mt-(C+a) \text{ and } F_a(t) = F(t)-at> M t^2-(at+C)  \text{ for } t \in \R^+,
\end{align}
where $C=C(M)$.
\end{remark}

For $g \in L^1(\RN) \cap L^{\infty}(\RN)$ and $a \ge 0$, we consider the following functionals on $\hst$:
\begin{align*}
   N_a(u) = \intRn g F_a(u), \text{ and } I_a(u) = \frac{1}{2} [u]_{s,2}^2 - N_a(u). 
\end{align*}
Using the upper bound of $F_a$ (Remark \ref{bound}) it follows that $N_a$ and $I_a$ are well-defined. Moreover, one can
also verify that $N_a, I_a \in \C^1(\hst, \R)$ and $N_a'(u), I_a'(u)$ for $u \in \hst$ are given by
\begin{align}\label{energy functional}
   N_a'(u)(v) = \intRn gf_a(u)v, \text{ and }  
    I_a'(u)(v) = \iint\limits_{\RN \times \RN}\frac{(u(x)-u(y)(v(x) - v(y))}{|x-y|^{N+2s}}\,\dx\dy -  N_a'(u)(v),
    \end{align}
where $v \in \hst$. Every critical point of $I_a$ corresponds to a solution of \eqref{SP}. In the following proposition, we prove the compactness of $N_a$ and $N_a'$. 

\begin{proposition}\label{compact map}
Let $N>2s$ and $\ga \in (2,2^*_s)$. Let $g \in L^1(\RN) \cap L^{\infty}(\RN)$. Assume that $f$ satisfies {\rm \ref{f1}}. Then the following hold for $a \ge 0$: 
\begin{enumerate}
    \item[{\rm(i)}] The functional $N_a$ is compact on $\hst$. Moreover, if $u_n \rightharpoonup u$ in $\hst$ and $a_n \ra a$ in $\R^+$, then $N_{a_n}(u_n) \ra N_a(u)$.
    \item[{\rm(ii)}] The map $N_a': \hst \ra (\hst)'$ is compact. Moreover, if $u_n \rightharpoonup u$ in $\hst$ and $a_n \ra a$ in $\R^+$, then $N_{a_n}'(u_n)(v) \ra N_{a}'(u)(v)$ for every $v \in \hst$. 
\end{enumerate}
\end{proposition}

\begin{proof}
(i) Let $u_n \rightharpoonup u$ in $\hst$. We show that $N_a(u_n) \ra N_a(u)$. The idea of the proof is similar to Proposition \ref{compact embeddings 2}. 
Set 
$L = \sup \left\{\norm{u_n}_{2^*_s}^{\ga} + \norm{u}_{2^*_s}^{\ga} + \norm{u_n}_{2^*_s} + \norm{u}_{2^*_s} : n \in \N \right\}$, which is finite since $(u_n)$ is bounded in $\hst$.   
Let $\ep > 0$ be given. Let $p$ be the conjugate exponent of $\frac{2^*_s}{\ga}$. 
 Using the density of $\C_c(\RN)$ into $L^{p}(\RN)$ and $L^{\frac{N}{2s}}(\RN)$, we take $g_\ep \in C_c(\RN)$ satisfying $$\abs{g_\ep}< \abs{g} \, \text{ and } \, \norm{g-g_{\ep}}_p + \norm{g-g_{\ep}}^{\frac{1}{2}}_{\frac{N}{2s}}<\frac{\ep}{L}.$$ For $a \ge 0$, we write
\begin{align}\label{compact 1}
    \abs{N_a(u_n)-N_a(u)} \le \intRn \abs{g-g_{\ep}}|F_a(u_n)-F_a(u)| + \intRn |g_{\ep}||F_a(u_n)-F_a(u)|.
\end{align}
Using \eqref{growth 2} and \eqref{cembed2.2}, the first integral of \eqref{compact 1} can be estimated as
\begin{align}\label{compact2}
    \intRn |g-g_{\ep}|&|F_a(u_n)-F_a(u)|  \le \intRn |g-g_{\ep}|\left(|F_a(u_n)| + |F_a(u)| \right) \no \\
    & \le C \intRn |g-g_{\ep}| \big ( |u_n|^{\ga} + |u|^{\ga} + \abs{u_n}+ \abs{u} \big) \no \\ 
    & \le C \left( \norm{ g - g_{\ep}}_p \left( \norm{u_n}_{2^*_s}^{\ga} + \norm{u}_{2^*_s}^{\ga} \right) +  \left( \norm{g-g_{\ep}}_{\frac{N}{2s}} \norm{g-g_{\ep}}_1 \right)^{\frac{1}{2}} \left(\norm{u_n}_{2^*_s} + \norm{u}_{2^*_s} \right) \right) \no \\
    & \le C \frac{\ep}{L} \left( \norm{u_n}_{2^*_s}^{\ga} + \norm{u}_{2^*_s}^{\ga} + \norm{u_n}_{2^*_s} + \norm{u}_{2^*_s}  \right) < C \ep,
\end{align}
where $C=C(f,a)$. Next, we show that the second integral of \eqref{compact 1} converges to zero as $n \ra \infty$. Let $K$ be the support of $g_{\ep}$. Since $\hst$ is compactly embedded into $L^{\ga}_{loc}(\RN)$ (Proposition \ref{compact embeddings 1}), $u_n \rightarrow u$ in $L^{\ga}(K)$. In particular, up to a subsequence, $u_n(x) \rightarrow u(x)$ a.e. in $K$. From the continuity of $F_a$, $F_a(u_n(x)) \ra F_a(u(x))$ a.e. in $K$. Further, since $|F_a(u_n)| \le C (|u_n|^{\ga} + |u_n|)$, and $\int_{K}|u_n|^{\ga} \ra \int_K |u|^{\ga}, \int_{K} \abs{u_n} \ra \int_{K} \abs{u}$, using the generalized dominated convergence theorem, $F_a(u_n) \ra F_a(u)$ in $L^1(K)$. Thus
\begin{align*}
    \intRn |g_{\ep}||F_a(u_n)-F_a(u)| \le \norm{g_{\ep}}_{\infty} \int_{K} |F_a(u_n)-F_a(u)| \ra 0, \text{ as } n \ra \infty.
\end{align*}
Therefore, from \eqref{compact 1}, $N_a(u_n) \ra N_a(u)$, as $n \ra \infty$. Now for a sequence $(a_n)$, we write 
\begin{align}\label{compact3}
   \abs{N_{a_n}(u_n) - N_a(u)} \le \intRn \abs{g} \abs{F_{a_n}(u_n)- F_a(u_n)} + \intRn \abs{g} \abs{F_a(u_n)-F_{a}(u)}.
\end{align}
By the compactness of $N_a$, the second integral of \eqref{compact3} converges to zero. Further, by noting that $F_{a_n}(u_n)-F_a(u_n) =(a-a_n) u_n$, the first integral of \eqref{compact3} can be estimated as
\begin{align}\label{compact4}
    \intRn \abs{g} \abs{F_{a_n}(u_n)- F_a(u_n)} \le |a_n-a| \intRn \abs{g} \abs{u_n} \le |a_n-a| \norm{g}_1^{\frac{1}{2}} \left( \intRn \abs{g} u_n^2 \right)^{\frac{1}{2}}.
\end{align}
Thus, combining \eqref{compact3} and \eqref{compact4}, and using Proposition \ref{compact embeddings 2} we get $N_{a_n}(u_n) \ra N_a(u)$, as $n \ra \infty$.

\noi (ii) Compactness of $N_a'$ similarly follows using the splitting arguments shown above. Proof of the convergence of $(N_{a_n}'(u_n)(v))$ holds similarly.
\end{proof}

Now we prove that the energy functional $I_a$ satisfies all the conditions of the mountain pass theorem. It is worth mentioning that $I_a$ may not satisfy the Palais-Smale condition due to the weaker Ambrosetti-Rabinowitz type nonlinearities for $f$ in \ref{f1}. Nevertheless, in the first two parts of the following proposition, $I_a$ satisfies a weaker Palais-Smale condition, called the Cerami condition introduced in \cite{Cerami1978}.

\begin{proposition}\label{I_a properties}
  Let $N>2s$ and $\ga \in (2,2^*_s)$. Let $f$ satisfies {\rm \ref{f1}} and {\rm \ref{f2}}. Let $g \in L^1(\RN) \cap L^{\infty}(\RN)$ be positive. Then the following hold:
  \begin{enumerate}
      \item[{\rm(i)}] Let $(u_n)$ be a bounded sequence in $\hst$ such that $I_a(u_n) \rightarrow c$ in $\R$ and $I_a'(u_n) \rightarrow 0$ in $(\hst)'$. Then $(u_n)$ has a convergent subsequence in $\hst$.
      \item[{\rm(ii)}] For any $c\in \R$, there exist $\eta, \be, \rho >0$ such that $\norm{I_a'(u)}[u]_{s,2} \ge \be$ for $u \in I_a^{-1} \left([c-\eta,c+\eta] \right)$ with $[u]_{s,2} \ge \rho$.
      \item[{\rm(iii)}] There exist $\rho>0, \be= \be(\rho)>0$, and $a_1=a_1(\rho)>0$ such that if $a \in (0, a_1)$, then $I_a(u) \ge \be$ for every $u \in \hst$ satisfying $[u]_{s,2}=\rho$. 
      \item[{\rm(iv)}] There exists $\tilde{u} \in \hst$ with $[\tilde{u}]_{s,2} > \rho$ such that $I_a(\tilde{u})<0$.
  \end{enumerate}
\end{proposition}

\begin{proof}
    (i) By the reflexivity, up to a subsequence $u_n \rightharpoonup u$ in $\hst$. We consider the functional $J(v) = [v]_{s,2}$, for $v \in \hst$. From \eqref{energy functional}, $J'(u_n)(u_n-u) = I_a'(u_n)(u_n-u) + N_a'(u_n)(u_n-u).$ By the hypothesis, $\abs{\left<I_a'(u_n), u_n-u \right>} \le \norm{I_a'(u_n)} [u_n-u]_{s,2} \ra 0$ as $n \ra \infty$. Moreover, since $N_a'$ is compact (Proposition \ref{compact map}), $N_a'(u_n)(u_n-u) \ra 0$ as $n \ra \infty$. Therefore, $J'(u_n)(u_n-u) \ra 0$ as $n \ra \infty$. Further, since $u_n \rightharpoonup u$ in $\hst$ and $J \in \C^1(\hst, \R)$, we also have $J'(u)(u_n-u) \ra 0.$ Therefore, $[u_n-u]_{s,2}^2 = J'(u_n)(u_n-u) - J'(u)(u_n-u)  \ra 0$, as required.

    \noi (ii) Our proof adapts the arguments given in \cite[Proposition 3.6]{BDS23}. On the contrary, assume that $(u_n)$ is a sequence in $\hst$ satisfying
\begin{align}\label{assumption}
    I_a(u_n) \ra c, [u_n]_{s,2} \ra \infty, \text{ and } \norm{I_a'(u_n)}[u_n]_{s,2} \ra 0, \text{ as } n \ra \infty.
\end{align}
Set $w_n= \frac{u_n}{[u_n]_{s,2}}$. Then $[w_n]_{s,2}=1$, and by the reflexivity, up to a subsequence, $w_n \rightharpoonup w$ in $\hst$. 

\noi  \textbf{Step 1:} This step shows that $w^+ = 0$ a.e. in $\RN$. We consider the set $\Om= \left\{x \in \RN : w(x)>0 \right\}$. On the contrary, we assume $|\Om|>0$. Using  \eqref{assumption}, $ [u_n]_{s,2}^2 - N_{a}'(u_n)(u_n) = I_a'(u_n)(u_n) \ra 0$. Hence for each $n \in \N$, we have 
\begin{align}\label{measure01}
    1 =  \frac{1}{[u_n]_{s,2}^2} \left( \intRn gf_a(u_n)u_n + \ep_n \right), 
\end{align}
where $\ep_n \ra 0$ as $n \ra \infty$. From the compactness of $\hst \hookrightarrow L^2_{loc}(\RN)$ and by the  Egorov's theorem, there exists $\Om_1 \subset \Om$ with $|\Om_1|>0$ such that $(w_n)$ converges to $w$ uniformly on $\Om_1$. Thus there exists $n_1 \in \N$ such that for $n \ge n_1$, $w_n \ge 0$ and hence $u_n \ge 0$ a.e. on $\Om_1$. This implies that, for every $n \ge n_1$,  $\Om_1 \subset \Om^+_n$ where $\Om^+_n:= \left\{ x \in \RN : u_n(x) \ge 0 \right\}.$ From the definition of $f_a$,  
$f_a(u_n)u_n = -au_n \ge 0$ on $\RN \setminus \Om^+_n$.  Therefore, using \eqref{measure01} and using the lower bound of $f_a$ in \eqref{growth 3}, for all $n \ge n_1$ we obtain  
\begin{align}\label{measure02}
    1  \ge \frac{1}{[u_n]_{s,2}^2} \left\{ \int_{\Om^+_n} gf_a(u_n)u_n + \ep_n \right\} 
    & \ge M \int_{\Om^+_n} g \frac{u_n^2}{[u_n]_{s,2}^2} -  \frac{(C_M + a)}{[u_n]_{s,2}} \int_{\Om^+_n} g \frac{u_n}{[u_n]_{s,2}} + \frac{\ep_n}{[u_n]_{s,2}^2} \no \\
    & \ge M \int_{\Om_1} g w_n^2 - \frac{(C_M + a)}{[u_n]_{s,2}} \int_{\Om^+_n} g w_n + \frac{\ep_n}{[u_n]_{s,2}^2}.
\end{align}
Further, $\int_{\Om_1} gw_n^2 \ra \int_{\Om_1} gw^2$ (by Proposition \ref{compact embeddings 2}) and $\int_{\Om^+_n} g w_n \le \int_{\RN} g w_n \le C(N,s) \left( \norm{g}_1 \norm{g}_{\frac{N}{2s}} \right)^{\frac{1}{2}}$. We take the limit as $n \ra \infty$ in \eqref{measure02} and using $[u_n]_{s,2} \ra \infty$, to obtain  
\begin{align*}
    1 \ge M \int_{\Om_1} gw^2, \text{ for arbitrarily large } M>0,
\end{align*}
which is a contradiction. Thus, $|\Om|=0$. \\
\noi  \textbf{Step 2:} For a fixed $n \in \N$, we set $m_n := \sup \left\{ I_a(tw_n): 0 \le t \le [u_n]_{s,2} \right\}$. Since the map $t \mapsto I_a(tw_n)$ is continuous on $[0,[u_n]_{s,2}]$, there exists $t_n \in [0,[u_n]_{s,2}]$ such that $m_n=I_a(t_n w_n)$. First, we claim that $m_n \ra \infty,$ as $n \ra \infty$. Since the sequence $(u_n)$ is unbounded, there exists $n_2 \in \N$ so that for $n \ge n_2$, $[u_n]_{s,2} \ge M$. Hence by definition, $m_n \ge I_a(Mw_n)$. Using the compactness of $N_a$ (Proposition \ref{compact map}) and $|\Om| = 0$ (Step 1), we get 
\begin{align*}
    \lim_{n \ra \infty} I_a(Mw_n) = \frac{M^2}{2} - \lim_{n \ra \infty} \intRn gF_a(Mw_n)
     = \frac{M^2}{2} - \intRn gF_a(Mw) 
     = \frac{M^2}{2} + aM\int_{\Om^c} gw.
\end{align*}
Since the quantity $\frac{M^2}{2} + aM\int_{\Om^c} gw$ is sufficiently large, we conclude that $I_a(Mw_n) \ra \infty$, as $n \ra \infty$, and hence the claim holds. Next, for each $n \in \N$,
\begin{align}\label{PS4}
    I_a(t_nw_n)-I_a(u_n) = \frac{t_n^2- [u_n]^2_{s,2}}{2} + \intRn g \left(F_a(u_n) - F_a(t_nw_n) \right). 
\end{align}
Set $s_n=\frac{t_n}{[u_n]_{s,2}}$. Then $s_n \in [0,1]$ and $s_n u_n=t_nw_n$. Clearly, $F_a(u_n(x)) - F_a(s_nu_n(x)) =0$ whenever $u_n(x)=0$. If $u_n(x) \neq 0$, then for $R>0$ given in \ref{f2}  we apply \cite[Proposition 3.3]{BDS23} to get 
\begin{align*}
    F_a(u_n(x)) - F_a(s_nu_n(x))  \le \frac{1-s_n^2}{2} u_n(x)f_a(u_n(x)) +C(R).
\end{align*}
Therefore, \eqref{PS4} yields
\begin{align*}
    I_a(t_n w_n)-I_a(u_n) & \le \frac{1-s_n^2}{2} \left( -[u_n]_{s,2}^2 + \intRn gu_nf_a(u_n) \right) + C(R) \norm{g}_1 \\
    & = \frac{s_n^2-1}{2}  I_a'(u_n)(u_n) + C(R) \norm{g}_1.
\end{align*}
Hence, in view of \eqref{assumption}, the sequence $(I_a(t_nw_n)-I_a(u_n))$ is bounded. On the other hand, individually $(I_a(u_n))$ is bounded (see \eqref{assumption}) and  $(I_a(t_nw_n))$ is unbounded, resulting in a contradiction. Therefore, such an unbounded sequence $(u_n)$ in \eqref{assumption} does not exist.

\noi (iii) Let $\ep>0$ be such that $\ep B_g \norm{g}_{\frac{N}{2s}} < \frac{1}{2}$, where $B_g$ is the best constant of \eqref{Sobolev Embedding}. Then using \eqref{growth 1} and the embeddings $\hst \hookrightarrow L^{\ga}(\RN,|g|)$ (Proposition \ref{compact embeddings 2}) we get
\begin{align}\label{MP1}
    I_a(u) & \ge \frac{[u]_{s,2}^2}{2} - \ep \intRn gu^2 - C \intRn g|u|^{\ga} - a\intRn g\abs{u} \no \\
    & \ge [u]_{s,2}^2 \left( \frac{1}{2} - \ep B_g \norm{g}_{\frac{N}{2s}} - C [u]_{s,2}^{\ga-2} \right) - aC[u]_{s,2},
\end{align}
where $C=C(f,g,N,s)$. Taking $[u]_{s,2} = \rho$, $I_a(u) \ge A(\rho) - aC\rho$, where $A(\rho)= C \rho^2 (1 - C_1 \rho^{\ga -2})$ for some constants $C,C_1$ independent of $a$. Let $\rho_1$ be the first nontrivial zero of $A$. For $\rho < \rho_1$, fix $a_1 \in (0, \frac{A(\rho)}{C\rho})$ and $\be=A(\rho) - a_1C\rho$. Therefore, using \eqref{MP1}, $I_a(u) \ge \be$ for every $a \in (0, a_1)$.

\noi (iv) We consider $\phi \in \C^2(\RN), \phi \ge 0$, and $[\phi]_{s,2}=1$. For $M,t>0$, using  \eqref{growth 3} we get $F_a(t \phi) > M(t\phi)^2-(at\phi+C(M))$. Hence 
\begin{align*}
    I_a(t \phi) \le t^2 \left( \frac{1}{2} - M \intRn g \phi^2 \right) + at \intRn g \phi + C(M) \norm{g}_1.
\end{align*}
Choose $M>\left( 2\int_{\RN} g \phi^2 \right)^{-1}$. Then $I_a(t \phi) \ra -\infty$, as $t \ra \infty$, i.e., there exists $t_1>\rho$ so that $I_a(t \phi)<0$ for $t > t_1$. Thus $\tilde{u} = t\phi$ with $t>t_1$ is the required function.
\end{proof}

\section{Existence, uniform boundedness, and regularity of the solutions}

In this section, we study the existence of solutions to \eqref{SP} and their various properties. This section contains the proof of Theorem \ref{existence}-\ref{regularity}. 

\noi \textbf{Proof of Theorem \ref{existence}:}
Recall $a_1, \beta, \tilde{u}$ as given in (iii) and (iv) of Proposition \ref{I_a properties}. For $a \in (0, a_1)$ using Proposition~\ref{I_a properties} and the fact that $[\cdot]_{s,2}$ is an equivalent norm in $\hst$ (from \eqref{Sobolev Embedding}), we observe that all the properties of the mountain pass theorem in \cite[Theorem~2.1]{S1991} are verified. Therefore, applying \cite[Theorem~2.1]{S1991} there exists $u_a \in \hst$ satisfying
\begin{align}\label{MP2} 
    I_a(u_a) =\inf_{\ga \in \Ga_{\tilde{u}}} \max_{s \in [0,1]} I_a(\ga(s)) \ge \be, \text{ and } I_a'(u_a)=0,
\end{align}
where $\Ga_{\tilde{u}} := \left\{ \ga \in \C([0,1], \hst) : \ga(0) = 0 \text{ and } \ga(1) = \tilde{u} \right\}$. Thus, $u_a$ is a nontrivial solution of \eqref{SP}. First, we show that the set $\{ I_a(u_a): a \in (0,a_1)\}$ is uniformly bounded. Define $\tilde{\ga} : [0,1] \ra \hst$ by $\tilde{\ga}(s) = s\tilde{u}$, where $\tilde{u}=t\phi$ for some $t>t_1$. Clearly, $\tilde{\ga} \in \Gamma_{\tilde{u}}$ and hence using \eqref{MP2} for $a \in (0,a_1)$,
\begin{align}\label{MP3}
    I_a(u_a) \le \max_{s \in [0,1]} I_a(\tilde{\ga}(s)) = \max_{s \in [0,1]} I_a(st\phi).
\end{align}
Further, since $F_a(st\phi) \ge M(st\phi)^2-a_1st\phi-C(M)$ (see \eqref{growth 3}) where $M \int_{\RN} g \phi^2 > \frac{1}{2}$, we get
\begin{align*}
    \max_{s \in [0,1]} I_a(st\phi) & \le \max_{s \in [0,1]} \left( s^2 t^2 \left( \frac{1}{2} - M \int_{\RN} g \phi^2 \right)  + st a_1 \int_{\RN} g \phi + C(M) \norm{g}_1\right) \\
    & \le t a_1 \int_{\RN} g \phi + C(M) \norm{g}_1.
\end{align*}
Thus from \eqref{MP3}, it is evident that $I_a(u_a) \le C$ for all $a \in (0,a_1)$. Next, we prove the existence of $a_2 \in (0,a_1)$ such that the set $\{ [u_a]_{s,2} : a\in (0,a_2)\}$ is uniformly bounded, i.e., $[u_a]_{s,2} \le C$ for all $a \in (0,a_2)$ and for some $C$. On the contrary, assume that no such $a_2$ and $C$ exist. Then there exists a sequence $(a_n)$ in $(0,a_1)$, such that $ a_n \ra 0,$ and $[u_{a_n}]_{s,2} \ra \infty,$ as $n \ra \infty.$
Observe that $I_{a_n}'(u_{a_n}) = 0$ for each $n \in \N$, and up to a subsequence, $I_{a_n}(u_{a_n}) \ra c$ in $\R$. Set $w_{a_n}= u_{a_n} {[u_{a_n}]^{-1}_{s,2}}$. Suppose $w_{a_n} \rightharpoonup w$ in $\hst. $ Now using the convergence $N_{a_n}(w_{a_n}) \ra N_0(w)$ (by Proposition \ref{compact map}-(i)), we can proceed with the same arguments as given in the proof of Proposition \ref{I_a properties}-(ii) (with $a$ replaced by $a_n$) to get the following contradiction:
\begin{align*}
    I_{a_n}(t_{a_n}w_{a_n})
    \le C(R) \norm{g}_1 + I_{a_n}(u_{a_n}), \; \forall \, n \in \N, \text{ and } I_{a_n}(t_{a_n}w_{a_n}) \ra \infty, \text{ as } n \ra \infty.
\end{align*}
Thus there exists $C$ such that $[u_a]_{s,2} \le C$ for all $a \in (0,a_2)$. Therefore, $(u_a)$ is uniformly bounded in $\hst$. \qed

Now we discuss the regularity of the mountain pass solution $u_a$. Before proceeding to the proof of Theorem~\ref{regularity}, we recall a result in \cite[Theorem~1.1]{No2021}, where the author provided a sufficient condition for H\"{o}lder regularity of weak solutions to a class of nonlocal equations. To state the result,  we define the following spaces
\begin{align*}
    & L_{2s}^1(\RN) := \left\{ u \in L_{loc}^1(\RN) : \intRn \frac{\abs{u(x)}}{1 + \abs{x}^{N+2s}} \, \dx < \infty  \right\}, \text{ and } \\
    & W_{loc}^{s,2}(\RN) := \left\{ u \in L_{loc}^2(\RN) : \left( \; \iint\limits_{K \times K}\frac{(u(x)-u(y))^2}{|x-y|^{N+2s}}\,\dx\dy \right)^{\frac{1}{2}} < \infty \right\},  
\end{align*}
where $K \subset \RN$ is any relatively compact open set.  

\begin{proposition}\label{holder regularity}
Let $s \in (0,1)$ and $N>2s$. Let $h \in L_{loc}^q(\RN)$ for $q> \frac{N}{2s}$. Assume that $u \in W_{loc}^{s,2}(\RN) \cap L_{2s}^1(\RN)$ is a weak solution of the equation $(-\De)^s u=h$ in $\RN$. Then $u \in C^{0, \al}_{loc}(\RN)$ for $\al \in (0, \min\{2s-\frac{N}{q},1\})$.
\end{proposition}

\noi \textbf{Proof of Theorem~\ref{regularity}:}  Let $a_2$ be as given in Theorem~\ref{existence} and $a \in (0,a_2)$. For $\tau >0$, 
we consider the truncation function $u_{\tau} \in L^{\infty}(\RN)$ associated with $u_a$; defined as $u_{\tau} = \max\{ -\tau, \min \{u_a, \tau \} \}.$ For $r \ge 2$, set $\phi= u_a \abs{u_{\tau}}^{r-2}$. Clearly, $\phi \in L^{\infty}(\RN) \cap \hst$. Taking $\phi$ as a test function in the weak formulation of $u_a$, we have 
\begin{align}\label{weak form}
    \iint\limits_{\RN \times \RN} &\frac{(u_a(x)-u_a(y)(u_a(x) \abs{u_{\tau}(x)}^{r-2} - u(y) \abs{u_{\tau}(y)}^{r-2})}{|x-y|^{N+2s}}\,\dx\dy \no \\
    &= \intRn g(x) f_a(u_a) u_a(x) \abs{u_{\tau}(x)}^{r-2} \, \dx.
\end{align}
We use \cite[Lemma 3.1]{IMS15} and the embedding \eqref{Sobolev Embedding} to estimate the L.H.S of \eqref{weak form} as
\begin{align*}
    & \iint\limits_{\RN \times \RN}\frac{(u_a(x)-u_a(y)(u_a(x) \abs{u_{\tau}(x)}^{r-2} - u_a(y) \abs{u_{\tau}(y)}^{r-2})}{|x-y|^{N+2s}}\,\dx\dy \\
    & \ge \frac{4(r-1)}{r^2} \iint\limits_{\RN \times \RN} \frac{(u_a(x) \abs{u_{\tau}(x)}^{\frac{r}{2}-1} - u_a(y) \abs{u_{\tau}(y)}^{\frac{r}{2}-1})}{|x-y|^{N+2s}} \,\dx\dy \\
    & \ge \frac{4(r-1)}{r^2} C(N,s) 
    \left( \intRn \left|u_a(x)\abs{u_{\tau}(x)}^{\frac{r}{2}-1}\right|^{2^*_s} \, \dx \right)^{\frac{2}{2^*_s}}.
\end{align*}
Hence from \eqref{weak form} we get for every $\tau>0$ that
\begin{align*}
    \left( \intRn \left|u_a(x)\abs{u_{\tau}(x)}^{\frac{r}{2}-1}\right|^{2^*_s} \, \dx \right)^{\frac{2}{2^*_s}} \le \frac{r^2}{4(r-1)} C(N,s) \intRn g(x) \abs{f_a(u_a)} \abs{u_a(x)}^{r-1} \, \dx.
\end{align*}
Letting $\tau \ra \infty$ the monotone convergence theorem yields 
\begin{align}\label{regular1}
     \left( \intRn \abs{u_a(x)}^{\frac{r}{2}2^*_s} \, \dx \right)^{\frac{2}{2^*_s}} 
    \le \frac{r^2}{4(r-1)} C(N,s) \intRn g(x) \abs{f_a(u_a)} \abs{u_a(x)}^{r-1} \, \dx.
\end{align}

\noi \textbf{Step 1}: In this step, for $r_1 = 2^*_s+1$, we show that $\abs{u_a}^{r_1} \in L^{\frac{2^*_s}{2}}(\RN)$ and there exists $C$ such that $\norm{u_a^{r_1}}_{\frac{2^*_s}{2}} \le C$ for all $a \in (0,a_2)$. Let $\ep>0$. Using the growth condition \ref{f1.1}, for every $a< a_2$, we have $\abs{f_a(u_a)} \le C + \ep \abs{u_a}^{2^*_s-1}$ where $C= C(\ep, a_2)$ (see \eqref{growth 2.1}). Applying the H\"{o}lder's inequality with the conjugate pair $(\frac{2^*_s}{2},\frac{2^*_s}{2^*_s-2})$ and uniform boundedness of $(u_a)$ in $\hst$ (from Theorem~\ref{existence}), we have the following estimates for all $a \in (0,a_2)$: 
\begin{align*}
    & \intRn g(x) \abs{u_a(x)}^{2^*_s} \, \dx \le \norm{g}_{\infty} \norm{u_a}_{\hst}^{2^*_s} \le C \norm{g}_{\infty}; \\
    & \intRn g(x) \abs{u_a(x)}^{2^*_s-2} \abs{u_a(x)}^{2^*_s+1} \, \dx  \le \norm{g}_{\infty} \norm{u_a}_{\hst}^{2^*_s-2} \left( \intRn \abs{u_a(x)}^{\frac{2^*_s}{2}(2^*_s+1)} \, \dx \right)^{\frac{2}{2^*_s}} \\
    & \quad \le C \norm{g}_{\infty} \left( \intRn \abs{u_a(x)}^{\frac{2^*_s}{2}(2^*_s+1)} \, \dx \right)^{\frac{2}{2^*_s}},
\end{align*}
where $C$ does not depend on $a$ and $\epsilon$. Hence \eqref{regular1} yields
\begin{align}\label{regular2}
    \left( \intRn \abs{u_a(x)}^{\frac{2^*_s}{2}r_1} \, \dx \right)^{\frac{2}{2^*_s}} 
    \le \frac{r_1^2}{4(r_1-1)} C \norm{g}_{\infty} \left( C(\ep,a_2) + \ep  \left( \intRn \abs{u_a(x)}^{\frac{2^*_s}{2}r_1} \, \dx \right)^{\frac{2}{2^*_s}} \right).
\end{align}
Now we choose $\ep$ such that 
\begin{align*}
    \left(\frac{r_1^2}{4(r_1-1)} C \norm{g}_{\infty} \right)\ep < \frac{1}{2}.
\end{align*}
Therefore, from \eqref{regular2}, there exists $C$ such that 
\begin{align}\label{unifromest}
    \frac{1}{2} \left( \intRn \abs{u_a(x)}^{\frac{2^*_s}{2}r_1} \, \dx \right)^{\frac{2}{2^*_s}} 
    \le \frac{r_1^2}{4(r_1-1)} C \norm{g}_{\infty}, \quad \forall \, a \in (0,a_2). 
\end{align}
Thus the set $ \{ \abs{u_a}^{r_1} : a \in (0,a_2) \}$ is uniformly bounded in $L^{\frac{2^*_s}{2}}(\RN)$.

\noi \textbf{Step 2}: In this step, we obtain the uniform $L^{\infty}$ bound of $(u_a)$. 
Using \eqref{growth 1} and \eqref{regular1}, for $r > r_1$  we have
\begin{align}\label{regular3}
    \left( \intRn \abs{u_a(x)}^{\frac{r}{2}2^*_s} \, \dx \right)^{\frac{2}{2^*_s}} 
    \le \frac{r^2}{4(r-1)} C(N,s,f,a_2) \intRn g(x) (1+ \abs{u_a(x)}^{2^*_s-1}) \abs{u_a(x)}^{r-1} \, \dx.
\end{align}
Set $m_1=\frac{2^*_s(2^*_s-1)}{r-2}$ and $m_2:= r-1-m_1$. Observe that $m_1<2^*_s$ whenever $r>r_1$. Applying Young's inequality with $(\frac{2^*_s}{m_1}, \frac{2^*_s}{2^*_s-m_1})$ we have the following estimate:
\begin{align*}
    \abs{u}^{r-1}= \abs{u}^{m_1}\abs{u}^{m_2} \le \frac{m_1}{2^*_s} \abs{u}^{2^*_s} + \frac{2^*_s - m_1}{2^*_s} \abs{u}^{\frac{2^*_sm_2}{2^*_s - m_1}} \le \abs{u}^{2^*_s} + \abs{u}^{\frac{2^*_sm_2}{2^*_s - m_1}},
\end{align*}
where we can verify $\frac{2^*_sm_2}{2^*_s - m_1} = 2^*_s -2+ r$. Hence by Theorem~\ref{existence},
\begin{align*}
    \intRn \abs{u_a(x)}^{r-1} \, \dx & \le \intRn \abs{u_a(x)}^{2^*_s} \, \dx + \intRn \abs{u_a(x)}^{2^*_s -2+ r} \, \dx \\
    & \le \norm{u_a}_{\hst}^{2^*_s} + \intRn \abs{u_a(x)}^{2^*_s -2+ r} \, \dx \\
    & \le C\left(1+ \intRn \abs{u_a(x)}^{2^*_s -2+ r} \, \dx\right),
\end{align*}
where $C$ does not depend on $a$. Further, notice that $\frac{r}{2(r-1)}< 1$ since $r>2$. Therefore, from \eqref{regular3}  we obtain the following estimate: 
\begin{align}\label{regular4}
    \left( 1+ \intRn \abs{u_a(x)}^{{\frac{r}{2}2^*_s}} \, \dx  \right)^{\frac{2}{2^*_s(r-2)}} \le r^{\frac{1}{r-2}}  \left( C\norm{g}_{\infty} \right)^{\frac{1}{r-2}}  
    \left( 1+ \intRn \abs{u_a(x)}^{2^*_s -2+ r} \, \dx  \right)^{\frac{1}{r-2}},
\end{align}
where $C=C(N,s,f,a_2)$. We consider the sequence $(r_j)$ defined as follows:
\begin{align*}
    r_1= 2^*_s+ 1, r_2 = 2 + \frac{2^*_s}{2}(r_1-2), \cdot \cdot \cdot,  r_{j+1} = 2 + \frac{2^*_s}{2}(r_j-2).  
\end{align*}
Notice that $2^*_s - 2 + r_{j+1} = \frac{2^*_s}{2} r_j$ and $r_{j+1} - 2 = \left( \frac{2^*_s}{2} \right)^j (r_1-2)$. Then \eqref{regular4} yields
\begin{align*}
    \left( 1+ \intRn \abs{u_a(x)}^{{\frac{r_{j+1}}{2}2^*_s}} \, \dx  \right)^{\frac{2}{2^*_s(r_{j+1}-2)}} \le \left( r_{j+1} C \norm{g}_{\infty} \right)^{\frac{1}{r_{j+1}-2}}
     \left( 1+ \intRn \abs{u_a(x)}^{\frac{2^*_s}{2} r_j} \, \dx  \right)^{\frac{2}{2^*_s(r_{j}-2)}}.
\end{align*}
Set $D_j =  \left( 1+ \intRn \abs{u_a(x)}^{\frac{2^*_s}{2} r_j} \, \dx  \right)^{\frac{2}{2^*_s(r_{j}-2)}}$. We iterate the above inequality to get 
\begin{align}\label{regular5}
    D_{j+1} \le \displaystyle C^{\sum_{k=2}^{j+1} \frac{1}{r_k - 2}} \left( \prod_{k=2}^{j+1} r_k^{\frac{1}{r_{k} - 2}} \right) D_1,
\end{align}
where $C= C(N,s,f,g,a_2)$. Using \eqref{unifromest} of Step 1, $D_1 \le C$ for some $C$ independent of $a$. Moreover, 
\begin{align*}
    D_{j+1} \ge  \left( \left(\intRn u_a(x)^{\frac{2^*_s r_{j+1}}{2}} \, \dx \right)^{\frac{2}{2^*_s r_{j+1}}} \right)^{\frac{r_{j+1}}{r_{j+1}-2}} = \norm{u_a}_{L^{\frac{2^*_s r_{j+1}}{2}}(\RN)}^{\frac{r_{j+1}}{r_{j+1}-2}}.
\end{align*}
Therefore, from \eqref{regular5} it is evident that 
\begin{align}\label{regular6}
    \norm{u_a}_{L^{\frac{2^*_s r_{j+1}}{2}}(\RN)}^{\frac{r_{j+1}}{r_{j+1}-2}} \le \displaystyle C^{\sum_{k=2}^{j+1} \frac{1}{r_k - 2}} \left( \prod_{k=2}^{j+1} r_k^{\frac{1}{r_{k} - 2}} \right)  C, \quad \forall \, a \in (0,a_2).
\end{align}
By noting that $r_j \ra \infty$ as $j \ra \infty$, we use the above estimate and the interpolation arguments to get $u_a \in L^r(\RN)$ for any $r \in [2_s^*, \infty)$, and $\norm{u_a}_r \le C(r, N, s, f, g, a_2)$ for all $a \in (0,a_2)$. Furthermore,
\begin{align*}
  \sum_{k=2}^\infty\frac{1}{r_k-2}= \frac{N}{2s(2^*_{s} - 2)}, \text{ and }
  \prod_{k=2}^{\infty}r_{k}^{\frac{1}{r_k - 2}} = \exp \left(\frac{2}{(2^*_s - 2)^2} \log \left(2 \left( \frac{2^*_s(2^*_s-2)}{2} \right)^{2^*_s} \right) \right).
\end{align*}
Thus taking the limit as $j \ra \infty$ in \eqref{regular6}, we conclude that $\norm{u_a}_{\infty} \le C(N, s, f, g, a_2)$ for all $a \in (0,a_2)$. 

\noi \textbf{Step 3}: This step verifies the continuity of $u_a$. Now, $u_a \in L^{\infty}(\RN) \subset L^1_{2s}(\RN)$. For $q> \frac{N}{2s}$,  
\begin{align*}
    \intRn (g f_a(u_a))^q \le C \intRn g^q(1 + \abs{u_a}^{(2^*_s-1)q}) \le C \norm{g}^q_q (1+ \norm{u_a}_{\infty}^{(2^*_s-1)q}) \le C, \quad \forall \, a \in (0,a_2).
\end{align*}
Further, using Proposition \ref{compact embeddings 1}, $\hst \hookrightarrow L_{loc}^2(\RN)$, and hence $\hst \hookrightarrow W_{loc}^{s,2}(\RN)$. Therefore, applying Proposition \ref{holder regularity} we conclude that $u_a \in C^{0, \al}_{loc}(\RN)$ for $\al \in (0, \min\{2s-\frac{N}{q},1\})$. In particular, $u_a \in \C(\RN)$ for all $a \in (0,a_2)$. This completes the proof. \qed  

Next, we prove a uniform lower bound for $(u_a)$ in $L^{\infty}(\RN)$.

\begin{proposition}\label{lowerbound}
Let $f,g,a_2,u_a$ be as given in Theorem~\ref{regularity}. Then there exist $\tilde{a}_2 \in (0,a_2)$ and $\be_1 > 0$ such that $\norm{u_a}_{\infty} \ge \be_1$, for all $a \in (0,\tilde{a}_2)$.
\end{proposition}

\begin{proof}
By the definition, $F_a(t) \ge -a|t|$, for all $t \in \R$. For $\beta$ as given in Proposition \ref{I_a properties}  we see that $I_a(u_a) \ge \be$, for all $a \in (0,a_2)$. Hence using the uniform boundedness of $(u_a)$ in $\hst$ (Theorem~\ref{existence}), we get for all $a < a_2$, 
\begin{align*}
    \frac{[u_a]_{s,2}^2}{2} = I_a(u_a) + \intRn gF_a(u_a) \ge \be - a\intRn g|u_a| \ge \be - aC \left(\norm{g}_1 \norm{g}_{\frac{N}{2s}} \right)^{\frac{1}{2}}.
\end{align*}
Choose $0 < \tilde{a}_2 < \min \left\{ \be C^{-1} \left(\norm{g}_1 \norm{g}_{\frac{N}{2s}} \right)^{-\frac{1}{2}}, a_2 \right\}$. Then 
$$ \frac{[u_a]_{s,2}^2}{2}  \ge \be_0 :=\be - \tilde{a}_2 C\left(\norm{g}_1 \norm{g}_{\frac{N}{2s}}\right)^{\frac{1}{2}}>0, \quad \forall \, a \in (0,\tilde{a}_2).$$ 
Hence using $|f_a(u_a)| \le C (1+ |u_a|^{2^*_s-1})$ and the fact that $u_a \in L^{\infty}(\RN)$ (Theorem~\ref{regularity}), we get
\begin{align*}
    \be_0 \le \frac{1}{2}\intRn g \abs{f_a(u_a)u_a} \le C \norm{g}_1 \left( \norm{u_a}_{\infty} + \norm{u_a}^{2^*_s}_{\infty} \right), 
\end{align*}
where $C$ does not depend on $a$. Therefore, there exists $\be_1 >0$ such that $\norm{u_a}_{\infty} \ge \be_1$, for all $ a \in (0,\tilde{a}_2)$.
\end{proof}

\section{Positivity of the solutions}
This section contains the proof of Theorem~\ref{result}. Afterwards, we give an example of a function satisfying all the hypotheses in this paper. The idea of our proof for the positivity of solutions is motivated by \cite[Theorem~1.1]{AHS20} (also, see  \cite[Theorem~4.14]{BDS23}).

\noi \textbf{Proof of Theorem~\ref{result}:} (i) For $\tilde{a_2}$ as in Proposition \ref{lowerbound}, take $a_3 \in (0, \tilde{a_2})$.  Let $(a_n)$ be a sequence in $(0,a_3)$ such that $a_n \ra 0$ as $n \ra \infty$. We aim to show that $u_{a_n}$ is nonnegative on $\RN$ for large $n$.  
By Theorem~\ref{existence}, $u_{a_n} \in \hst$ is a mountain pass solution of \eqref{SP}, such that the following hold (up to a subsequence):
\begin{align}\label{PS1}
    I_{a_n}'(u_{a_n}) = 0, \text{ for each } n \in \N, \; I_{a_n}(u_{a_n}) \ra c, \text{ as } n \ra \infty, \text{ and } \norm{u_{a_n}}_{s,2} \le C. 
\end{align}
Therefore, $(u_{a_n})$ is a bounded Palais-Smale sequence in $\hst$. For brevity, we denote the sequence $(u_{a_n})$ by $(u_n)$. Since $I_a$ satisfies the Cerami condition, using the same arguments as in Proposition \ref{I_a properties}-(i) (replacing $a$ by $a_n$), we obtain that (up to a subsequence) $u_n \ra \tilde{u}$ in $\hst$, and  $u_n(x) \ra \tilde{u}(x)$ a.e. in $\RN$. We split the rest of our proof into two steps. In the first step, we prove that  $\tilde{u}$ is nonnegative and $(u_n)$ converges uniformly to $\tilde{u}$ on $\RN$. In the second step, we obtain the non-negativity of $u_n$.

\noi  \textbf{Step 1:} We consider the following function:
\begin{align*}
    f_0(t) = \left\{\begin{array}{ll} 
            f(t) , & \text {if }  t \ge 0; \\ 
            0, & \text{if} \; t \le 0.  \\
             \end{array} \right.
\end{align*} 
Since $a_n \ra 0$ in $\R^+$ and $u_n \ra \tilde{u}$ in $\hst$, using (ii) of Proposition \ref{compact map} it follows that 
\begin{align*}
    \lim_{n \ra \infty} \intRn g(x)f_{a_n}(u_n) \phi(x) \, \dx = \intRn g(x) f_0(\tilde{u}) \phi(x) \, \dx, \quad \forall \, \phi \in \hst. 
\end{align*}
So we have the following identity for every $\phi \in \hst$:
\begin{align*}
    \iint\limits_{\RN \times \RN}\frac{(\tilde{u}(x)- \tilde{u}(y)(\phi(x) - \phi(y))}{|x-y|^{N+2s}}\,\dx\dy & = \lim_{n \ra \infty} \iint\limits_{\RN \times \RN}\frac{(u_n(x)- u_n(y)(\phi(x) - \phi(y))}{|x-y|^{N+2s}}\,\dx\dy \\
    & = \lim_{n \ra \infty} \intRn g(x)f_{a_n}(u_n) \phi(x) \, \dx = \intRn g(x) f_0(\tilde{u}) \phi(x) \, \dx.
\end{align*}
Therefore, $\tilde{u}$ satisfies the following equation weakly:
\begin{equation}\label{P2}
\begin{aligned} 
 (-\De)^s u =  g(x)f_0(u) \text{ in } \RN.
 \end{aligned}
 \end{equation}
Since  $\abs{x}^{2s-N}$ is a fundamental solution of $(-\De)^s$ (see \cite[Theorem 5]{St2019}), we get 
 \begin{align}\label{Riesz1}
     \tilde{u}(x) = C(N,s) \intRn \frac{g(y)f_0(\tilde{u}(y))}{|x-y|^{N-2s}} \, \dy \ge 0 \; \text{ a.e. in } \RN.
 \end{align}
Further, using the similar set of arguments as given in Theorem~\ref{regularity}, $\tilde{u} \in L^{\infty}(\RN) \cap \C(\RN)$. Moreover, since $u_n$ is a solution of \eqref{SP}, we also have
\begin{align}\label{Riesz2}
    u_n(x) = C(N,s) \intRn \frac{g(y)f_{a_n}(u_n(y))}{|x-y|^{N-2s}} \, \dy \; \text{ a.e. in } \RN.
\end{align}
Using \eqref{Riesz1} and \eqref{Riesz2} we estimate $\abs{u_n - \tilde{u}}$ as follows: 
\begin{align}\label{positive1}
    &\abs{u_n(x) - \tilde{u}(x)} \le \no \\
    &C(N,s) \left( \int_{B_1(x)} g(y) \frac{\abs{f_{a_n}(u_n(y)) - f_0(\tilde{u}(y))}}{|x-y|^{N-2s}} \, \dy +  \int_{\RN \setminus B_1(x)} g(y) \frac{\abs{f_{a_n}(u_n(y)) - f_0(\tilde{u}(y))}}{|x-y|^{N-2s}} \, \dy \right).
\end{align}
Take $1< \de < \frac{N}{N-2s}$. Applying the H\"{o}lder's inequality with the conjugate pair $(\de, \de')$ we estimate the first integral of \eqref{positive1} as
\begin{align*}
    \int_{B_1(x)} g(y) \frac{\abs{f_{a_n}(u_n(y)) - f_0(\tilde{u}(y))}}{|x-y|^{N-2s}} \, \dy \le & \left( \int_{B_1(x)} \frac{1}{|x-y|^{(N-2s)\de}}  \, \dy \right)^{\frac{1}{\de}} \\
    & \left(  \int_{B_1(x)} g(y)^{\de'} \abs{f_{a_n}(u_n(y)) - f_0(\tilde{u}(y))}^{\de'}  \, \dy \right)^{\frac{1}{\de'}}. 
\end{align*}
We calculate $\int_{B_1(x)} \frac{1}{|x-y|^{(N-2s)\de}}  \, \dy = \om_N \int_0^1 \frac{r^{N-1}}{r^{(N-2s)\de}} \, \dr\le C(N).$ We show that the second integral of the above inequality converges to zero. Observe that   
\begin{align}\label{positive2}
    \abs{f_{a_n}(u_n(y)) - f_0(\tilde{u}(y))}^{\de'} \le 2^{\de'-1}\left( a_n^{\de'} + \abs{f_0(u_n(y)) - f_0(\tilde{u}(y))}^{\de'} \right),
\end{align}
where using \eqref{growth 2} and the uniform boundedness of the mountain pass solution in $L^{\infty}(\RN)$ (Theorem~\ref{regularity}), we get $\abs{f_0(u_n(y)) - f_0(\tilde{u}(y))}^{\de'} \le C2^{\de'-1}(1+ \abs{u_n(y)}^{(2^*_s-1)\de'} + \abs{\tilde{u}(y)}^{(2^*_s-1)\de'})\le C$. Moreover, $f_0(u_n(y)) \ra f_0(\tilde{u}(y)$ and $a_n \ra 0$ as $n \ra \infty$. Therefore, by the dominated convergence theorem, $\int_{B_1(x)} g(y)^{\de'} ( a_n^{\de'} + \abs{f_0(u_n(y)) - f_0(\tilde{u}(y))}^{\de'}) \, \dy \ra 0$. Hence using \eqref{positive2} and the generalized dominated convergence theorem, we conclude
\begin{align*}
    \int_{B_1(x)} g(y)^{\de'} \abs{f_{a_n}(u_n(y)) - f_0(\tilde{u}(y))}^{\de'}  \, \dy \ra 0, \text{ as } n \ra \infty. 
\end{align*}
Next, the second integral of \eqref{positive1} has the following bound: 
\begin{align*}
    \int_{\RN \setminus B_1(x)} g(y) \frac{\abs{f_{a_n}(u_n(y)) - f_0(\tilde{u}(y))}}{|x-y|^{N-2s}} \, \dy \le  \int_{\RN \setminus B_1(x)} g(y) \abs{f_{a_n}(u_n(y)) - f_0(\tilde{u}(y))}. 
\end{align*}
Again by the generalized dominated convergence theorem, $\int_{\RN \setminus B_1(x)} g(y) \abs{f_{a_n}(u_n(y)) - f_0(\tilde{u}(y))} \ra 0$. Therefore, \eqref{positive1} yields $u_n \ra \tilde{u}$ in $L^{\infty}(\RN)$ as $n \ra \infty$. Thus $(u_n)$ converges uniformly to $\tilde{u}$ on $\R^N$.

\noi \textbf{Step 2:} Now $\tilde{u} \in \C(\RN) \cap L^{\infty}(\RN)$ is a non-negative function and satisfies $(-\De)^s \tilde{u} \ge 0$ in the weak sense in $\RN$ (from \eqref{P2}). Suppose $\tilde{u}(x_0)=0$ for some $x_0 \in \RN$. Since $\tilde{u}$ satisfies all the properties of the strong maximum principle \cite[Proposition 5.2.1]{DMV17}, we conclude that $\tilde{u}$ vanishes identically on $\RN$. Further, from the uniform lower bound of $(u_n)$ in Proposition \ref{lowerbound} and the uniform convergence of $u_n \ra \tilde{u}$ (Step 1), there exists $\be_2>0$ such that $\norm{\tilde{u}}_{\infty} \ge \be_2$, a contradiction. Thus $\tilde{u} \neq 0$ on $\RN$ and \cite[Proposition 5.2.1]{DMV17} yields $\tilde{u} >0$ on $\RN$. Therefore, again from the uniform convergence of $(u_n)$, there exists $n_1 \in \N$ such that for all $n \ge n_1$, $u_n \ge 0$ on $\RN$.

\noi (ii) For a sequence $(a_n)$ given in (i), we show $u_{a_n}$ (denoted by $u_n$) is positive on $\RN$ for large $n$. For each $n \in \N$, since $f_{a_n}$ is locally Lipschitz (from \ref{f3}) and $0 \le u_n, \tilde{u} \le C$, we have $|f_{a_n}(u_n(y)) - f_{a_n}(\tilde{u}(y))| \le M |u_n(y)-\tilde{u}(y)|$ for some $M>0$. For $x\in \RN \setminus \{0\}$, using \eqref{Riesz1} and \eqref{Riesz2} we write
\begin{align*}
     \abs{u_n(x) - \tilde{u}(x)} & \le C(N,s) \left( M  \intRn \frac{g(y)|u_n(y)-\tilde{u}(y)|}{|x-y|^{N-2s}} \, \dy + a_n \intRn \frac{g(y)}{|x-y|^{N-2s}} \, \dy \right).
\end{align*}
Since $g$ satisfies \ref{g1}, from the above inequality and Step 1 we get
\begin{align*}
    \abs{u_n(x) - \tilde{u}(x)} \le C(N,s) \left( M \norm{u_n-\tilde{u}}_{\infty} + a_n \right) \frac{C(g)}{|x|^{N-2s}}.
\end{align*}
Hence 
\begin{align}\label{limit}
    \underset{{x \in \RN \setminus \{0\}}}{\sup}\left\{ |x|^{N-2s} \abs{u_n(x) - \tilde{u}(x)} \right\} \ra 0, \text{ as } n \ra \infty.
\end{align}
Now we show that $\underset{|x| \ra \infty}{\lim} |x|^{N-2s} \tilde{u}(x)>0$. Using \eqref{Riesz1} we get
\begin{align}\label{Riesz3}
   \lim_{|x| \ra \infty} |x|^{N-2s} \tilde{u}(x) & = C(N,s) \lim_{|x| \ra \infty} \intRn \frac{g(y)f_0(\tilde{u}(y))|x|^{N-2s}}{|x-y|^{N-2s}} \, \dy  \no \\
    & \ge C(N,s) \lim_{|x| \ra \infty}  \int_{B_R} \frac{g(y)f_0(\tilde{u}(y))|x|^{N-2s}}{|x-y|^{N-2s}} \, \dy, 
\end{align}
for any $R>0$. Choose $R>0$ arbitrarily. Then there exists $x \in \RN$ such that $|x|>2R+1$. Hence 
\begin{align*}
    |x-y|^{N-2s} \ge \abs{|x|-|y|}^{N-2s} \ge \abs{|x|-R}^{N-2s} \ge 2^{2s-N} \left(1+ \abs{x} \right)^{N-2s}, \; \text{ for } y \in B_R.
\end{align*}
Using the above estimate, for $y \in B_R$ we get $$\frac{g(y)f_0(\tilde{u}(y))|x|^{N-2s}}{|x-y|^{N-2s}} \le 2^{N-2s} g(y)f_0(\tilde{u}(y)).$$
Further, $\frac{g(y)f_0(\tilde{u})|x|^{N-2s}}{|x-y|^{N-2s}} \ra g(y)f_0(\tilde{u})$ a.e. in $B_R$, as $|x| \ra \infty$. Therefore, the dominated convergence theorem yields 
\begin{align*}
    \lim_{|x| \ra \infty}  \int_{B_R} \frac{g(y)f_0(\tilde{u}(y))|x|^{N-2s}}{|x-y|^{N-2s}} \, \dy = \int_{B_R} g(y)f_0(\tilde{u}(y)) \, \dy.
\end{align*}
Hence from \eqref{Riesz3} we conclude that 
\begin{align*}
    \lim_{|x| \ra \infty} |x|^{N-2s} \tilde{u}(x) \ge C(N,s) \int_{B_R} g(y)f_0(\tilde{u}(y)) \, \dy.
\end{align*}
Letting $R \ra \infty$ and applying the Fatou's lemma, $\lim_{|x| \ra \infty} |x|^{N-2s} \tilde{u}(x) \ge C(N,s) \int_{\R^N} g(y)f_0(\tilde{u}(y)) \, \dy.$ Further, using \eqref{Riesz1} and $\tilde{u}>0$ on $\RN$,
it follows that $gf_0(\tilde{u}) \gneqq 0$ on $\R^N$. Hence,  $\lim_{|x| \ra \infty} |x|^{N-2s} \tilde{u}(x) >0.$ Therefore, from \eqref{limit} there exists $n_2 \in \N$ and $R > > 1$ such that for $n \ge n_2$, $u_n >0$ on $B_R^c$. Moreover, since $\tilde{u} \in \C(\RN)$, there exists $\eta> 0$ such that $\tilde{u} > \eta$ on $\overline{B_R}$. Therefore, from the uniform convergence of $(u_n)$ (in Step 1), there exists $n_3 \in \N$ such that for $n \ge n_3$, $u_n >0$ on $\overline{B_R}$. Thus, by choosing $n_4=\max\{n_2,n_3\}$, we see that for $n \ge n_4$, $u_n>0$ on $\RN$. This completes the proof. \qed 

\begin{example}\label{example1}
Let $s \in (0,1)$ and $N>2s$. For $R>0$, we consider the following functions 
\begin{align*}
    f(t)=2t\ln (1+|t|), \text{ for } t \in \R^+; \quad g(y)= \frac{\chi_{B_R(0)}(y)}{(1+\abs{y})^{2(N-2s)}}, \text{ for } y \in \RN.
\end{align*}
(i) We can verify that $f$ satisfies \ref{f1}-\ref{f3} and \ref{f1.1}. \\
\noi (ii) Clearly $g \in L^1(\RN) \cap L^{\infty}(\RN)$. We show that $g$ satisfies \ref{g1}. For $x \in \RN \setminus \{0\}$, split 
\begin{align*}
    \int_{\RN} \frac{g(y)}{\abs{x-y}^{N-2s}} \, \dy = \int_{\abs{x-y}\ge \frac{\abs{x}}{2}} \frac{g(y)}{\abs{x-y}^{N-2s}} \, \dy +  \int_{\abs{x-y}\le \frac{\abs{x}}{2}} \frac{g(y)}{\abs{x-y}^{N-2s}} \, \dy.
\end{align*}
The first integral has the following bound: 
\begin{align*}
    \int_{\abs{x-y}\ge \frac{\abs{x}}{2}} \frac{g(y)}{\abs{x-y}^{N-2s}} \, \dy \le \left( \frac{2}{\abs{x}} \right)^{N-2s} \norm{g}_{1}.
\end{align*}
Now consider the case $\abs{x-y}\le \frac{\abs{x}}{2}$. Set $z=x-y$. Then $\abs{x-z}\ge \abs{ \abs{x} - \abs{z} } \ge \frac{\abs{x}}{2}$ and hence $\abs{x-z} \ge \abs{z}$. Using the fact that $g(y) \le g(\frac{x}{2})$ and $g(y) \le g(z)$ we obtain 
\begin{align*}
    \int_{\abs{x-y}\le \frac{\abs{x}}{2}} \frac{g(y)}{\abs{x-y}^{N-2s}} \, \dy \le \int_{\abs{z} \le \frac{\abs{x}}{2}} \frac{(g(\frac{x}{2}) g(z))^{\frac{1}{2}}}{\abs{z}^{N-2s}} \, \dz & \le \frac{2^{N-2s}\chi_{B_R(0)}(\frac{x}{2})}{(2+\abs{x})^{N-2s}} \int_{B_R(0)} \frac{\dz}{(1+\abs{z})^{N-2s} \abs{z}^{N-2s}} \\
    & \le \left( \frac{2}{\abs{x}} \right)^{N-2s} \omega(N) \int_0^R \frac{r^{2s-1}}{(1+r)^{N-2s}}  \, \dr \\
    &\le \left( \frac{2}{\abs{x}} \right)^{N-2s} C(N),
\end{align*}
for some constant $C(N)$. Here $\omega(N)$ is the measure of $B_1(0)$ in $\RN$. Therefore, $$\abs{x}^{N-2s} \int_{\RN} \frac{g(y)}{\abs{x-y}^{N-2s}} \, \dy \le C(g, N)$$ for $x \in \RN \setminus \{0\}$. 
\end{example}

\bibliographystyle{plainurl}

\begin{thebibliography}{10}

\bibitem{AHS20}
C.~O. Alves, A.~R.~F. de~Holanda, and J.~A. dos Santos.
\newblock Existence of positive solutions for a class of semipositone problem
  in whole {$\mathbb{R}^N$}.
\newblock {\em Proc. Roy. Soc. Edinburgh Sect. A}, 150(5):2349--2367, 2020.
\newblock \href {https://doi.org/10.1017/prm.2019.20}
  {\path{doi:10.1017/prm.2019.20}}.

\bibitem{AR1973}
A.~Ambrosetti and P.~H. Rabinowitz.
\newblock Dual variational methods in critical point theory and applications.
\newblock {\em J. Functional Analysis}, 14:349--381, 1973.
\newblock \href {https://doi.org/10.1016/0022-1236(73)90051-7}
  {\path{doi:10.1016/0022-1236(73)90051-7}}.

\bibitem{Barbi_2022}
G.~Barbi, D.~Capacci, A.~Chierici, L.~Chirco, V.~Giovacchini, and
  S.~Manservisi.
\newblock A numerical approach to the fractional laplacian operator with
  applications to quasi-geostrophic flows.
\newblock {\em Journal of Physics: Conference Series}, 2177(1):012013, 2022.
\newblock URL: \url{https://dx.doi.org/10.1088/1742-6596/2177/1/012013}, \href
  {https://doi.org/10.1088/1742-6596/2177/1/012013}
  {\path{doi:10.1088/1742-6596/2177/1/012013}}.

\bibitem{BDS23}
N.~Biswas, U.~Das, and A.~Sarkar.
\newblock On the fourth order semipositone problem in {$\Bbb R^N$}.
\newblock {\em Discrete Contin. Dyn. Syst.}, 43(1):411--434, 2023.
\newblock \href {https://doi.org/10.3934/dcds.2022154}
  {\path{doi:10.3934/dcds.2022154}}.

\bibitem{BLP14}
L.~Brasco, E.~Lindgren, and E.~Parini.
\newblock The fractional {C}heeger problem.
\newblock {\em Interfaces Free Bound.}, 16(3):419--458, 2014.
\newblock \href {https://doi.org/10.4171/IFB/325} {\path{doi:10.4171/IFB/325}}.

\bibitem{BK1979}
H.~Br\'{e}zis and T.~Kato.
\newblock Remarks on the {S}chr\"{o}dinger operator with singular complex
  potentials.
\newblock {\em J. Math. Pures Appl. (9)}, 58(2):137--151, 1979.

\bibitem{BS83}
K.~J. Brown and R.~Shivaji.
\newblock Simple proofs of some results in perturbed bifurcation theory.
\newblock {\em Proc. Roy. Soc. Edinburgh Sect. A}, 93(1-2):71--82, 1982/83.
\newblock \href {https://doi.org/10.1017/S030821050003167X}
  {\path{doi:10.1017/S030821050003167X}}.

\bibitem{BV2016}
C.~Bucur and E.~Valdinoci.
\newblock {\em Nonlocal diffusion and applications}, volume~20 of {\em Lecture
  Notes of the Unione Matematica Italiana}.
\newblock Springer, [Cham]; Unione Matematica Italiana, Bologna, 2016.
\newblock \href {https://doi.org/10.1007/978-3-319-28739-3}
  {\path{doi:10.1007/978-3-319-28739-3}}.

\bibitem{CFL16}
A.~Castro, D.~G. de~Figueredo, and E.~Lopera.
\newblock Existence of positive solutions for a semipositone {$p$}-{L}aplacian
  problem.
\newblock {\em Proc. Roy. Soc. Edinburgh Sect. A}, 146(3):475--482, 2016.
\newblock \href {https://doi.org/10.1017/S0308210515000657}
  {\path{doi:10.1017/S0308210515000657}}.

\bibitem{CHS95}
A.~Castro, M.~Hassanpour, and R.~Shivaji.
\newblock Uniqueness of non-negative solutions for a semipositone problem with
  concave nonlinearity.
\newblock {\em Comm. Partial Differential Equations}, 20(11-12):1927--1936,
  1995.
\newblock \href {https://doi.org/10.1080/03605309508821157}
  {\path{doi:10.1080/03605309508821157}}.

\bibitem{CS88}
A.~Castro and R.~Shivaji.
\newblock Nonnegative solutions for a class of nonpositone problems.
\newblock {\em Proc. Roy. Soc. Edinburgh Sect. A}, 108(3-4):291--302, 1988.
\newblock \href {https://doi.org/10.1017/S0308210500014670}
  {\path{doi:10.1017/S0308210500014670}}.

\bibitem{CS89}
A.~Castro and R.~Shivaji.
\newblock Nonnegative solutions for a class of radially symmetric nonpositone
  problems.
\newblock {\em Proc. Amer. Math. Soc.}, 106(3):735--740, 1989.
\newblock \href {https://doi.org/10.2307/2047429} {\path{doi:10.2307/2047429}}.

\bibitem{Cerami1978}
G.~Cerami.
\newblock An existence criterion for the critical points on unbounded
  manifolds.
\newblock {\em Istit. Lombardo Accad. Sci. Lett. Rend. A}, 112(2):332--336
  (1979), 1978.

\bibitem{CDS15}
M.~Chhetri, P.~Dr\'{a}bek, and R.~Shivaji.
\newblock Existence of positive solutions for a class of {$p$}-{L}aplacian
  superlinear semipositone problems.
\newblock {\em Proc. Roy. Soc. Edinburgh Sect. A}, 145(5):925--936, 2015.
\newblock \href {https://doi.org/10.1017/S0308210515000220}
  {\path{doi:10.1017/S0308210515000220}}.

\bibitem{CI2017}
P.~Constantin and M.~Ignatova.
\newblock Remarks on the fractional {L}aplacian with {D}irichlet boundary
  conditions and applications.
\newblock {\em Int. Math. Res. Not. IMRN}, (6):1653--1673, 2017.
\newblock \href {https://doi.org/10.1093/imrn/rnw098}
  {\path{doi:10.1093/imrn/rnw098}}.

\bibitem{CQT17}
D.~G. Costa, H.~Ramos~Q., and H.~Tehrani.
\newblock A variational approach to superlinear semipositone elliptic problems.
\newblock {\em Proc. Amer. Math. Soc.}, 145(6):2661--2675, 2017.
\newblock \href {https://doi.org/10.1090/proc/13426}
  {\path{doi:10.1090/proc/13426}}.

\bibitem{DZ05}
E.~N. Dancer and Z.~Zhang.
\newblock Critical point, anti-maximum principle and semipositone
  {$p$}-{L}aplacian problems.
\newblock {\em Discrete Contin. Dyn. Syst.}, pages 209--215, 2005.

\bibitem{DT21}
R.~Dhanya and S.~Tiwari.
\newblock A multiparameter fractional {L}aplace problem with semipositone
  nonlinearity.
\newblock {\em Commun. Pure Appl. Anal.}, 20(12):4043--4061, 2021.
\newblock \href {https://doi.org/10.3934/cpaa.2021143}
  {\path{doi:10.3934/cpaa.2021143}}.

\bibitem{DMV17}
S.~Dipierro, M.~Medina, and E.~Valdinoci.
\newblock {\em Fractional elliptic problems with critical growth in the whole
  of {$\Bbb{R}^n$}}, volume~15 of {\em Appunti. Scuola Normale Superiore di
  Pisa (Nuova Serie) [Lecture Notes. Scuola Normale Superiore di Pisa (New
  Series)]}.
\newblock Edizioni della Normale, Pisa, 2017.
\newblock \href {https://doi.org/10.1007/978-88-7642-601-8}
  {\path{doi:10.1007/978-88-7642-601-8}}.

\bibitem{DV2015}
S.~Dipierro and E.~Valdinoci.
\newblock A density property for fractional weighted {S}obolev spaces.
\newblock {\em Atti Accad. Naz. Lincei Rend. Lincei Mat. Appl.},
  26(4):397--422, 2015.
\newblock \href {https://doi.org/10.4171/RLM/712} {\path{doi:10.4171/RLM/712}}.

\bibitem{GH2015}
P.~Gatto and J.~S. Hesthaven.
\newblock Numerical approximation of the fractional {L}aplacian via
  {$hp$}-finite elements, with an application to image denoising.
\newblock {\em J. Sci. Comput.}, 65(1):249--270, 2015.
\newblock \href {https://doi.org/10.1007/s10915-014-9959-1}
  {\path{doi:10.1007/s10915-014-9959-1}}.

\bibitem{IMS15}
A.~Iannizzotto, S.~Mosconi, and M.~Squassina.
\newblock {$H^s$} versus {$C^0$}-weighted minimizers.
\newblock {\em NoDEA Nonlinear Differential Equations Appl.}, 22(3):477--497,
  2015.
\newblock \href {https://doi.org/10.1007/s00030-014-0292-z}
  {\path{doi:10.1007/s00030-014-0292-z}}.

\bibitem{JS2010}
J.~Shi J.~Jiang.
\newblock {\em Bistability dynamics in some structured ecological models, in:
  S. Cantrell, C. Cosner, S. Ruan (Eds.), Spatial Ecology}.
\newblock Chapman \& Hall/CRC Mathematical and Computational Biology Series.
  CRC Press, Boca Raton, FL, 2009.

\bibitem{LLV2022}
E.~Lopera, C.~L{\'o}pez, and R.~Vidal.
\newblock Existence of positive solutions for a parameter fractional $ p
  $-laplacian problem with semipositone nonlinearity.
\newblock {\em arXiv preprint arXiv:2211.02790}, 2022.

\bibitem{No2021}
S.~Nowak.
\newblock Higher {H}\"{o}lder regularity for nonlocal equations with irregular
  kernel.
\newblock {\em Calc. Var. Partial Differential Equations}, 60(1):Paper No. 24,
  37, 2021.
\newblock \href {https://doi.org/10.1007/s00526-020-01915-1}
  {\path{doi:10.1007/s00526-020-01915-1}}.

\bibitem{OSS02}
S.~Oruganti, J.~Shi, and R.~Shivaji.
\newblock Diffusive logistic equation with constant yield harvesting. {I}.
  {S}teady states.
\newblock {\em Trans. Amer. Math. Soc.}, 354(9):3601--3619, 2002.
\newblock \href {https://doi.org/10.1090/S0002-9947-02-03005-2}
  {\path{doi:10.1090/S0002-9947-02-03005-2}}.

\bibitem{SB1951}
E.~E. Salpeter and H.~A. Bethe.
\newblock A relativistic equation for bound-state problems.
\newblock {\em Phys. Rev. (2)}, 84:1232--1242, 1951.

\bibitem{S1991}
M.~Schechter.
\newblock A variation of the mountain pass lemma and applications.
\newblock {\em J. London Math. Soc. (2)}, 44(3):491--502, 1991.
\newblock \href {https://doi.org/10.1112/jlms/s2-44.3.491}
  {\path{doi:10.1112/jlms/s2-44.3.491}}.

\bibitem{St2019}
P.~R. Stinga.
\newblock User's guide to the fractional {L}aplacian and the method of
  semigroups.
\newblock In {\em Handbook of fractional calculus with applications. {V}ol. 2},
  pages 235--265. De Gruyter, Berlin, 2019.

\end{thebibliography}

\end{document}